\documentclass[11pt]{amsart}
\usepackage[T1]{fontenc}
\usepackage{tikz, tikz-cd, enumerate, stmaryrd, amsmath, amsthm, amssymb,
hyperref, comment, todonotes}
\usepackage[shortlabels]{enumitem}
\usetikzlibrary{arrows}
\usetikzlibrary{positioning}
\usepackage[utf8x]{inputenc}
\newcommand{\bb}{\textbf}
\newcommand{\uu}{\underline}
\newcommand{\ol}{\overline}
\newcommand{\mc}{\mathcal}
\newcommand{\wh}{\widehat}
\newcommand{\wt}{\widetilde}
\newcommand{\mf}{\mathfrak}

\newcommand{\HH}{\mathbb{H}}
\newcommand{\ZZ}{\mathbb{Z}}

\newcommand{\RR}{\mathbb{R}}

\newcommand{\QQ}{\mathbb{Q}}

\newcommand{\FF}{\mathbb{F}}

\DeclareMathOperator{\Ind}{Ind}

\DeclareMathOperator{\Divv}{Div}

\DeclareMathOperator{\coker}{coker}
\DeclareMathOperator{\im}{im}
\DeclareMathOperator{\id}{id}

\DeclareMathOperator{\Span}{Span}
\DeclareMathOperator{\res}{res}

\DeclareMathOperator{\Sl}{Sl}
\DeclareMathOperator{\GCD}{GCD}
\DeclareMathOperator{\ord}{ord}
\DeclareMathOperator{\Spec}{Spec}

\DeclareMathOperator{\cha}{char}

\DeclareMathOperator{\Aut}{Aut}

\DeclareMathOperator{\mmod}{-mod}
\newtheorem{Theorem}{Theorem}[section]
\newtheorem*{mainthm}{Main Theorem}
\newtheorem{Lemma}[Theorem]{Lemma}
\newtheorem{Corollary}[Theorem]{Corollary}

\newtheorem{Proposition}[Theorem]{Proposition}

\newtheorem{Example}[Theorem]{Example}
\newtheorem{Question}[Theorem]{Question}
\newtheorem{Remark}[Theorem]{Remark}
\numberwithin{equation}{section}
\begin{document}

\title{Equivariant splitting of the Hodge--de Rham exact sequence}
\author[J. Garnek]{J\k{e}drzej Garnek}
\address{Graduate School, Adam Mickiewicz University\\
	Faculty of Mathematics and Computer Science\\
Umultowska 87, 61-614 Poznan, Poland}
\email{jgarnek@amu.edu.pl}
\date{}

\begin{abstract}
	Let $X$ be an algebraic curve with an action of a finite group~$G$ over a field $k$.
	We show that if the Hodge-de Rham short exact sequence of $X$ splits $G$-equivariantly
	then the action of $G$ on $X$ is weakly ramified. In particular, this generalizes the result of
	K\"{o}ck and Tait for hyperelliptic curves.
	We discuss also converse statements and tie this problem to lifting
	coverings of curves to the ring of Witt vectors of length $2$. 
\end{abstract}

\maketitle
\bibliographystyle{plain}

\section{Introduction} \label{sec:intro}

Let $X$ be a smooth proper algebraic variety over a field $k$. Recall that its de Rham cohomology
may be computed in terms of Hodge cohomology via the spectral sequence
\begin{equation} \label{eqn:hodge_de_rham_ss}
	E^{ij}_1 = H^j(X, \Omega^i_{X/k}) \Rightarrow H^{i+j}_{dR}(X/k).
\end{equation}
Suppose that the spectral sequence~\eqref{eqn:hodge_de_rham_ss} degenerates at the first page.
This is automatic if $\cha k = 0$. For a field of positive characteristic,
this happens for instance if $X$ is a smooth projective curve or an abelian variety,
or (by a celebrated result of Deligne and Illusie from \cite{DeligneIllusie_relevements})
if $\dim X > \cha k$ and $X$ lifts to $W_2(k)$, the ring of Witt vectors of length~$2$.
Under this assumption we obtain the following exact sequence:
\begin{equation} \label{eqn:hodge_de_rham_se}
	0 \to H^0(X, \Omega_{X/k}) \to H^1_{dR}(X/k) \to H^1(X, \mc O_X) \to 0.
\end{equation}
If $X$ is equipped with an action of a finite group $G$,
the terms of the sequence~\eqref{eqn:hodge_de_rham_se} become $k[G]$-modules.
In case when ${\cha k \nmid \# G}$, Maschke theorem allows one to conclude that the sequence~\eqref{eqn:hodge_de_rham_se}
splits equivariantly. However, this might not be true in case when $\cha k = p > 0$
and $p | \# G$, as shown in~\cite{KockTait2018}. The goal of this article is to show that
for curves the sequence~\eqref{eqn:hodge_de_rham_se} \emph{usually} does not split equivariantly.

Let $X$ be a curve over an algebraically closed field of characteristic $p > 0$
with an action of a finite group $G$. For $P \in X$, denote by $G_{P, n}$
the $n$-th ramification group of $G$ at $P$. Let also:
\begin{equation} \label{eqn:definition_of_nP}
	n_P := \max\{ n : G_{P, n} \neq 0 \}.
\end{equation}
Following \cite{Kock2004}, we say that the action of $G$ on $X$ is \bb{weakly ramified}
if $n_P \in \{ 0, 1 \}$ for every $P \in X$.
\begin{mainthm}
	Suppose that $X$ is a smooth projective curve over an algebraically closed field $k$ of
	a finite characteristic $p > 2$ with an action of a finite group~$G$. If 
	\begin{equation} \label{eqn:intro_decomposition}
	 H^1_{dR}(X/k) \cong H^0(X, \Omega_{X/k}) \oplus H^1(X, \mc O_X)
	\end{equation}
	as $k[G]$-modules then the action of $G$ on $X$ is weakly ramified.
\end{mainthm}
\noindent The example below is a direct generalization of results proven in~\cite{KockTait2018}.
\begin{Example} \label{ex:hyperelliptic1}
	Suppose that $k$ is an algebraically closed field of characteristic~$p$.
	Let $X/k$ be the smooth projective curve with the affine part given by the equation:
	\[ y^m = f(z^p - z), \]
	where $f$ is a separable polynomial and $p \nmid m$. Denote by $\mc P$ the set of points of~$X$ at infinity. One checks that $\# \mc P = \delta := \GCD(m, \deg f)$ (cf. \cite[Section~1]{Towse1996}).
	The group $G = \ZZ/p$ acts on $X$ via the automorphism $\varphi(z, y) = (z+1, y)$.
	In this case
	\begin{equation} \label{eqn:nP}
		n_P =
		\begin{cases}
			m/\delta, & \textrm{ if $P \in \mc P$,} \\
			0,                 & \textrm{ otherwise.}        
		\end{cases}
	\end{equation}
	(cf. Example~\ref{ex:hyperelliptic2}). Thus if the exact sequence~\eqref{eqn:hodge_de_rham_se}
	splits $G$-equivariantly, then by Main Theorem either $p = 2$, or $m|\deg f$.
\end{Example}
The main idea of the proof of Main Theorem is to compare $H^1_{dR}(X/k)^G$
and $H^1_{dR}(Y/k)$, where $Y := X/G$. The discrepancy between those groups
is measured by the sheafified version of group cohomology, introduced by Grothendieck in~\cite{GrothendieckQuelques}. This allows us to compute the 'defect'
\begin{eqnarray}
	\delta(X, G) &:=& \dim_k H^0(X, \Omega_{X/k})^G + \dim_k H^1(X, \mc O_X)^G 
	\label{eqn:defect_X} \\
	&-& \dim_k H^1_{dR}(X/k)^G \nonumber 
\end{eqnarray}
in terms of some local terms connected to Galois cohomology (cf. Proposition~\ref{prop:main_lemma}).
We compute these local terms in case of Artin-Schreier coverings, which
leads to the following theorem.
\begin{Theorem} \label{thm:invariants_AS}
	Suppose that $X$ is a smooth projective curve over an algebraically closed field $k$ of
	characteristic $p > 0$ with an action of the group~$G = \ZZ/p$. 
	Then:
	\begin{eqnarray*}
		\delta(X, G) = \sum_{P \in X} \left(\left[ \frac{(p-1) \cdot (n_P + 1)}{p} \right] - 1 - \left[ \frac{n_P - 1}{p} \right] \right).
	\end{eqnarray*}
\end{Theorem}
\noindent Theorem \ref{thm:invariants_AS} shows that if the group action of $G = \ZZ/p$ on a curve
is not weakly ramified and $p > 2$ then $\delta(X, G) > 0$. This immediately implies Main Theorem
for $G \cong \ZZ/p$. The general case may be easily derived from
this special one.\\ 

The natural question arises: {\sl to which extent is the converse of Main Theorem true?} We give some partial answers. In characteristic $2$, we were able to produce a counterexample
(cf. Subsection~\ref{subsec:counterexample}). We provide also some positive results. In particular, we prove the following theorem.
\begin{Theorem} \label{thm:weakly_then_invariants}
    If the action of $G$ on a smooth projective curve $X$ over an algebraically closed field $k$ is weakly ramified then
    the sequence
    \[ 0 \to H^0(X, \Omega_{X/k})^G \to H^1_{dR}(X/k)^G \to H^1(X, \mc O_X)^G \to 0 \]
    is exact also on the right.
\end{Theorem}
\noindent To derive Theorem~\ref{thm:weakly_then_invariants} we use the method of
proof of Main Theorem and a~result of K\"{o}ck from~\cite{Kock2004}.\\

We were also able to show that the equivariant decomposition~\eqref{eqn:intro_decomposition} holds
under some additional assumptions.
\begin{Theorem} \label{thm:converse_to_main_thm}
 Keep the above notation. If any of the following conditions is satisfied:
 \begin{enumerate}[(1)]
	\item the action of $G$ on $X$ is weakly ramified and the $p$-Sylow subgroup of $G$ is cyclic,
	
	\item the action of $G$ on $X$ lifts to $W_2(k)$,
	
	\item $X$ is ordinary,
\end{enumerate}
 then we have an isomorphism~\eqref{eqn:intro_decomposition} of~$\FF_p[G]$-modules.
\end{Theorem}
\noindent Parts (1), (2), (3) of Theorem~\ref{thm:converse_to_main_thm} are proven in Lemma~\ref{lem:cyclic_sylow}, Theorem~\ref{thm:deligne_illusie} and Corollary~\ref{cor:ordinary} respectively. In fact, we prove more precise statements,
involving the conjugate spectral sequence. This allows to prove that the conditions (2) and (3) of Theorem~\ref{thm:converse_to_main_thm} imply that the action of $G$ on $X$ is weakly ramified.

In order to prove (1) we use a description of modular representations of cyclic groups.
(2) and (3) are easy corollaries of the equivariant version of results of Deligne and Illusie from~\cite{DeligneIllusie_relevements}. The connection of \cite{DeligneIllusie_relevements} with the splitting of the Hodge-de Rham exact sequence was observed by Piotr Achinger.\\

\bb{Notation.} Throughout the paper we will use the following
notation (unless stated otherwise):
\begin{itemize}
 \item $k$ is an algebraically closed field of a finite characteristic $p$.
 \item $G$ is a finite group.
 \item $X$ is a smooth projective curve equipped with an action of $G$.
 \item $Y := X/G$ is the quotient curve, which is of genus $g_Y$.
 \item $\pi : X \to Y$ is the canonical projection.
 \item $R = \sum_{P \in X} d_P \cdot (P) \in \Divv(X)$ is the ramification divisor of $\pi$.
 
 \item $R' := \left[\frac{\pi_* R}{\# G} \right] \in \Divv(Y)$,
 where for $\delta \in \Divv(Y) \otimes_{\ZZ} \QQ$, we denote by $[\delta]$ the integral part taken
 coefficient by coefficient.
 \item $k(X)$, $k(Y)$ are the function fields of $X$ and $Y$.
 \item $\ord_Q(f)$ denotes the order of vanishing of a function $f$ at a point $Q$.
 \item $\uu{A}_X$ denotes the constant sheaf on $X$ associated to a ring $A$.
 \item $X'$ is the Frobenius twist of $X$.
 \item $F : X \to X'$ is the relative Frobenius of $X/k$.
\end{itemize}
Fix now a (closed) point $P$ in $X$. Denote:
 \begin{itemize}
  \item $G_{P, i}$ -- the $i$th ramification group of $\pi$ at $P$, i.e.
	\begin{equation*}
		G_{P, i} := \{ g \in G : g(f) \equiv f \pmod{\mf m_{X, P}^{i + 1}}
		\quad \textrm{ for all } f \in \mc O_{X, P} \}.
	\end{equation*}
	Note in particular that (since $k$ is algebraically closed) the inertia
	group $G_{P, 0}$ coincides with the decomposition group at $P$, i.e.
	the stabilizer of $P$ in $G$. Moreover, one has:
 \[ d_P = \sum_{i \ge 0} (\# G_{P, i} - 1). \]
	\item $e_P$ -- the ramification index of $\pi$ at $P$, i.e.
	$e_P = \# G_{P, 0}$.
	
	\item $n_P$ is given by the formula~\eqref{eqn:definition_of_nP}.
 \end{itemize} 
Also, by abuse of notation, for $Q \in Y$ we write $e_Q := e_P$,
$d_Q := d_P$, $n_Q := n_P$ for any $P \in \pi^{-1}(Q)$. Note that these quantities don't
depend on the choice of $P$.\\

\bb{Outline of the paper.} Section~\ref{sec:group_cohomology}
presents some preliminaries on the group cohomology of sheaves. We focus on the sheaves coming from Galois coverings of a curve. We use this theory to express the 'defect' $\delta(X, G)$
as a sum of local terms coming from Galois cohomology of certain modules in Section~\ref{sec:main_lemma}. In Section~\ref{sec:local_computations} we compute these local terms for Artin-Schreier coverings, which allows us to prove of Main Theorem and Theorem~\ref{thm:invariants_AS}. In the final section we discuss the converse statements to Main Theorem and its relation to the problem of lifting curves with a given group action.
Also, we give an counterexample in characteristic $2$.
We include also Appendix, which allows to compute the dimensions of $H^0(X, \Omega_X)^G$,
$H^1(X, \mc O_X)^G$ and $H^1_{dR}(X/k)^G$.\\

\bb{Acknowledgements.}
The author wishes to express his gratitude to Wojciech Gajda and Piotr Achinger
for many stimulating conversations and their patience.
The author also thanks Bartosz Naskr\k{e}cki, who suggested elliptic curves in characteristic~$2$
as a rich source of examples. The author was supported by NCN research grant UMO-2017/27/N/ST1/00497 and by the doctoral scholarship of Adam Mickiewicz University. Part of the work was done during the author's stay during Simons Semester at IMPAN in Warsaw (November 2018), supported by the National Science Center grant 2017/26/D/ST1/00913, the grant 346300 for IMPAN from the Simons Foundation and the matching 2015-2019 Polish MNiSW fund. This paper is part of the author's PhD thesis.


\section{Review of group cohomology} \label{sec:group_cohomology}
Recall that our goal is to compare $H^1_{dR}(X/k)^G$ and $H^1_{dR}(Y/k)$, where $Y := X/G$. To this end, we need to work in the $G$-equivariant setting.
\subsection{Group cohomology of sheaves}
Let $A$ be any commutative ring and $G$ a finite group. We define the \bb{$i$-th group cohomology}, $H^i_A(G, -)$, as the $i$-th derived functor of the functor
\begin{equation*}
 (-)^G : A[G] \mmod \to A \mmod, \quad M \mapsto M^G := \{ m \in M : g \cdot m = m \}.
\end{equation*}
One checks that if $A \to B$ is a morphism of rings and
$M$ is a $B[G]$-module then $H^i_B(G, M)$ and $H^i_A(G, M)$ are
isomorphic $A$-modules for all $i \ge 0$ (cf. \cite[Lemma 0DVD]{stacks-project}).
In particular, $H^i_A(G, M)$ is isomorphic as a $\ZZ$-module
to the usual group cohomology ($H^i_{\ZZ}(G, M)$ in our notation).
Thus without ambiguity we will denote it by $H^i(G, M)$.
For a future use we note the following properties of group cohomology:
\begin{itemize}
 \item If $M = \Ind_H^G N$ is an induced module (which for finite groups
 is equivalent to being a coinduced module) then
 \begin{equation} \label{eqn:shapiro}
  H^i(G, M) \cong H^i(H, N),
 \end{equation}
 (this property is known as \bb{Shapiro lemma}, cf. \cite[Proposition VIII.2.1.]{Serre1979}).
 \item If $M$ is a $\FF_p[G]$-module and $G$ has a normal $p$-Sylow subgroup $P$ then: 
 \begin{equation} \label{eqn:sylow}
  H^i(G, M) \cong H^i(P, M)
 \end{equation}
 (for a proof observe that $H^i(G/P, N)$ is killed by multiplication by~$p$ for any $\FF_p[G]$-module $N$ and use \cite[Theorem IX.2.4.]{Serre1979} to obtain $H^i(G/P, N) = 0$ for $i \ge 1$. Then use Lyndon–Hochschild–Serre spectral sequence).
 
 \item Suppose that $A$ is a finitely generated algebra over a field $k$, which is a local ring with maximal ideal $\mf m$. If $M$ is a finitely generated $A$-module then
 \begin{equation} \label{eqn:completion}
	H^i(G, M) \cong H^i(G, \wh M_{\mf m}),
 \end{equation}
 where $\wh M_{\mf m}$ denotes the completion of $M$ with respect to $\mf m$
 (see e.g. proof of \cite[Lemme 3.3.1]{BertinMezard2000} for a brief justification).
\end{itemize}
The above theory extends to sheaves, as explained e.g. in \cite{GrothendieckQuelques}
and \cite{BertinMezard2000}. We briefly recall this theory.
Let $(Y, \mc O)$ be a ringed space and let $G$ be a finite group.
By an $\mc O[G]$-sheaf on $(Y, \mc O)$ we understand a
sheaf $\mc F$ equipped with an $\mc O$-linear action of $G$ on $\mc F(U)$ for every open
subset $U \subset Y$, compatible with respect to the restrictions. 
The $\mc O[G]$-sheaves form a category $\mc O[G] \mmod$, which is abelian and has enough injectives.
For any $\mc O[G]$-sheaf~$\mc F$ one may define a sheaf $\mc F^G$
by the formula
\begin{equation*}
	U \mapsto \mc F(U)^G := \{ f \in \mc F(U) : \forall_{g \in G} \quad g \cdot f = f \}.
\end{equation*}
We denote the $i$-th derived functor of
\[ (-)^G : \mc O[G] \mmod \to \mc O \mmod \]
by $\mc H^i_{(Y, \mc O)}(G, -)$. Similarly as in the case of
modules, one may neglect the dependence on the sheaf $\mc O$ and
write simply $\mc H^i(G, M)$. If $\mc F = \wt M$ is a quasicoherent $\mc O[G]$-module
coming from a $\mc O(Y)[G]$-module $M$, one may compute the group cohomology of sheaves via
the standard group cohomology:
\begin{equation*}
	\mc H^i(G, \mc F) \cong \wt{H^i(G, M)}.
\end{equation*}
In particular, group cohomology of a quasicoherent $\mc O[G]$-sheaf is
a quasicoherent $\mc O$-module. Moreover for any $Q \in Y$:
\begin{equation} \label{eqn:group_cohomology_stalk}
	\mc H^i(G, \mc F)_Q \cong H^i(G, \mc F_Q).
\end{equation}
The sheaf group cohomology may be also computed using \v{C}ech complex (cf. \cite[section 3.1]{BertinMezard2000}). However, we will not use this fact in any way.

\subsection{Galois coverings of curves}
We now turn to the case of curves over a field $k$. 
Let $X/k$ be a smooth projective curve with an action of a finite group~$G$, i.e. a homomorphism $G \to \Aut_k(X)$. In this case one can define the quotient $Y := X/G$ of $X$ by the $G$-action.
It is a smooth projective curve. Its underlying space is the topological quotient $X/G$ and the structure sheaf is given by $\pi_*^G(\mc O_X)$, where $\pi : X \to Y$ is the quotient morphism.
We say that $X$ is a \bb{$G$-covering of~$Y$}. \\

In this section we will investigate the $G$-sheaves on $Y$ coming from 
its $G$-coverings. Suppose that $\pi : X \to Y$ is a $G$-covering of~$Y$.
Let~$\mc F$ be a $\mc O_X$-sheaf on~$X$ with a $G$-action lifting that on~$X$.
Then $\pi_* \mc F$ is an $\mc O_Y [G]$-module. It is natural to try to relate the group cohomology
of $\pi_* \mc F$ to the ramification of~$\pi$. Suppose for a while that the action of $G$ on $X$ is free, i.e. that $\pi : X \to Y$ is unramified. In this case the functors
\begin{eqnarray*}
	\mc F &\mapsto& \pi_*^G(\mc F)\\
	\pi^*(\mc G) &\mapsfrom& \mc G
\end{eqnarray*}
are exact and provide an equivalence between the category of coherent $\mc O_Y$-modules and
coherent $\mc O_X$-modules (cf. \cite[Proposition II.7.2, p. 70]{MumfordAV}).
In particular, $\mc H^i(G, \pi_* \mc F) = 0$ for all $i \ge 1$ and every coherent $\mc O_X$-module
$\mc F$. The following Proposition treats the general case.
\begin{Proposition} \label{prop:gr_coh_torsion_sheaf}
	Keep the notation introduced in Section~\ref{sec:intro}. Let 
	$\mc F$ be a coherent $\mc O_X$-module, which is $G$-equivariant. Then for
	every $i \ge 1$
	\[ \mc H^i(G, \pi_* \mc F)\]
	is a torsion sheaf, supported on the wild ramification locus of $\pi$.
\end{Proposition}
To prove Proposition~\ref{prop:gr_coh_torsion_sheaf} we shall need the following
lemma involving group cohomology of modules over Dedekind domains.
\begin{Lemma} \label{lem:group_cohomology_higher_ramification}
  Let $k$ be an algebraically closed field.
	Let $B$ be a finitely generated $k$-algebra, which is a Dedekind domain equipped with a $k$-linear action of the group $G$. Suppose that $A := B^G$ is a principal ideal domain with a maximal ideal $\mf q$. Let $G_{\mf p, i}$ denote the $i$-th higher ramification group of a prime ideal $\mf p \in \Spec B$ over $\mf q$. Then for every $B$-module $M$ we have an isomorphism
	of $B$-modules:
	\begin{equation*} 
		H^i(G, M) \cong H^i(G_{\mf p, 1}, M_{\mf p})
	\end{equation*}
	(here $M_{\mf p}$ denotes the localisation of $M$ at $\mf p$).
\end{Lemma}
\begin{proof}
 One easily sees that we have an isomorphism of $B[G]$-modules
	\[ \wh M_{\mf q} \cong \Ind_{G_{\mf p, 0}}^G \wh M_{\mf p} \]
 and thus by~\eqref{eqn:completion} and \eqref{eqn:shapiro} $H^i(G, M)
 \cong H^i(G_{\mf p, 0}, M_{\mf p})$. Moreover, $G_{\mf p, 1}$ is a normal $p$-Sylow subgroup of
 $G_{\mf p, 0}$ (cf. \cite[Corollary 4.2.3., p. 67]{Serre1979}). Hence the proof follows by~\eqref{eqn:sylow}.
\end{proof}
\begin{proof}[Proof of Proposition~\ref{prop:gr_coh_torsion_sheaf}]
 Denote by $\xi$ the generic point of $Y$. Recall that by the Normal Base Theorem (cf. \cite[sec. 4.14]{Jacobson_basic_I}), $k(X) = \Ind^G k(Y)$ is an induced
$G$-module. Therefore $(\pi_* \mc F)_{\xi}$ is also an induced $G$-module
(since it is a $k(X)$-vector space of finite dimension) and by~\eqref{eqn:shapiro}:
\begin{eqnarray*}
 \mc H^i(G, \pi_* \mc F)_{\xi} = H^i(G, (\pi_* \mc F)_{\xi}) = 0.
\end{eqnarray*}
 Thus, since the sheaf $\mc H^i(G, \pi_* \mc F)$ is coherent, it must be a torsion sheaf.
 Note that if a point $Q \in Y$ is tamely ramified then $G_{P, 1} = 0$ for any $P \in \pi^{-1}(Q)$
 and thus $\mc H^i(G, \pi_* \mc F)_Q = 0$ by Lemma~\ref{lem:group_cohomology_higher_ramification}. This concludes the proof.
\end{proof}
\noindent We will recall now a standard formula describing $G$-invariants of a $\mc O_Y[G]$-module
coming from an invertible $\mc O_X$-module. For a reference see e.g. the proof of \cite[Proposition 5.3.2]{BertinMezard2000}.
\begin{Lemma} \label{lem:invertible_sheaf_invariants}
	For any $G$-invariant divisor $D \in \Divv(X)$:
	\[ \pi_*^G(\mc O_X(D)) = \mc O_Y \left( \left[ \frac{\pi_* D}{\# G} \right] \right), \]
	where for $\delta \in \Divv(Y) \otimes_{\ZZ} \QQ$, we denote by $[\delta]$ the integral part taken
	coefficient by coefficient.
\end{Lemma}
\begin{Corollary} \label{cor:pi_G_omega}
Keep the notation of Section~\ref{sec:intro}. Let:
	\begin{equation*}
		R' = \left[ \frac{\pi_* R}{\# G} \right] \in \Divv(Y).
	\end{equation*}
	Then:
 \[ \pi_*^G \Omega_{X/k} = \Omega_{Y/k} \otimes \mc O_Y(R'). \]
	In particular:
 \[ \dim_k H^0(X, \Omega^1_{X/k})^G
			      =
			      \begin{cases}
			      	g_Y,               & \textrm{ if } R' = 0, \\
			      	g_Y - 1 + \deg R', & \textrm{ otherwise. } 
			      \end{cases}
	\]	
\end{Corollary}
\begin{proof}
  The first claim follows by Lemma~\ref{lem:invertible_sheaf_invariants} by taking $D$ to be the canonical divisor of $X$ and using the Riemann-Hurwitz formula. 
  To prove the second claim, observe that
	\begin{equation*}
		H^0(X, \Omega_{X/k})^G = H^0(Y, \pi_*^G \Omega_{X/k}) = 
		H^0(Y, \Omega_{Y/k} \otimes \mc O_Y(R'))
	\end{equation*}
	and apply the Riemann-Roch theorem (cf. \cite[Theorem IV.1.3]{Hartshorne1977}). 
\end{proof}
\noindent We end this section with one more elementary observation.
\begin{Lemma} \label{lem:vanishing_of_r_prim}
	$R'$ (given as above) vanishes if and only if the morphism
	$\pi : X \to Y$ is tamely ramified.
\end{Lemma}
\begin{proof}
	Recall that $R = \sum_{P \in X} d_P \cdot (P)$. Hence
	\begin{eqnarray*}
	 R' &=& \sum_{Q \in Y} \left[ \frac{d_Q \cdot \# \pi^{-1}(Q)}{\# G} \right] (Q)\\
	 &=& \sum_{Q \in Y} \left[ \frac{d_Q}{e_Q} \right] (Q).
	\end{eqnarray*}
	Note however that $d_Q \ge e_Q - 1$ with an equality if and only if 
	$\pi$ is tamely ramified at $Q$. This completes the proof.
\end{proof}
\section{Computing the defect} \label{sec:main_lemma}
\noindent The goal of this section is to prove the following Proposition.
\begin{Proposition} \label{prop:main_lemma}
  We follow the notation introduced in Section~\ref{sec:intro}. Then:
  \[ \delta(X, G) = \sum_{Q \in Y} \dim_k \im \bigg(H^1(G, (\pi_* \mc O_X)_Q)
  \to H^1(G, (\pi_* \Omega_{X/k})_Q) \bigg), \]
  where
  \[ H^1(G, (\pi_* \mc O_X)_Q) \to H^1(G, (\pi_* \Omega_{X/k})_Q) \]
  is the map induced by the differential $\mc O_X \to \Omega_{X/k}$.
\end{Proposition}

\subsection{Proof -- preparation}
Recall that $i$-th hypercohomology $\HH^i(Y, -)$
is defined as the $i$-th derived functor of
\[ H^0(Y, -) : \uu k_Y \mmod \to k \mmod. \]
The hypercohomology may be computed in terms of the usual cohomology using the spectral sequences:
\begin{equation}
	_{I} E^{ij}_1 = H^j(Y, \mc F^i) \Rightarrow \HH^{i+j}(Y, \mc F^{\bullet}),
	\label{eqn:spectral_sequence_hypercohomology}
\end{equation}
\begin{equation}
	_{II} E^{ij}_2 = H^i(Y, h^j(\mc F^{\bullet})) \Rightarrow \HH^{i+j}(Y, \mc F^{\bullet}).
	\label{eqn:conjugated_spectral_sequence_hypercohomology}
\end{equation}
The de Rham cohomology of $X$ is defined as the hypercohomology of
the de Rham complex $\Omega_{X/k}^{\bullet}$. Note that $\pi$ is an affine morphism. Therefore $\pi_*$ is an exact functor on the category of quasi-coherent sheaves. Thus using the spectral sequence
\eqref{eqn:spectral_sequence_hypercohomology} we obtain:
\[ H^i_{dR}(X/k) = \HH^i(X, \Omega_{X/k}^{\bullet}) = \HH^i(Y, \pi_* \Omega_{X/k}^{\bullet}). \]
We start with the following observation.
\begin{Lemma} \label{lem:degeneration}
 The spectral sequence
 \[ E^{ij}_1 = H^j(Y, \pi_*^G \Omega_{X/k}^i) \Rightarrow
 \HH^{i+j}(Y, \pi_*^G \Omega_{X/k}^{\bullet})  \]
 degenerates at the first page.
\end{Lemma}
\begin{proof}
 We have a morphism of complexes
 $\Omega_{Y/k}^{\bullet} \to \pi_*^G \Omega_{X/k}^{\bullet}$,
 which is an isomorphism on the zeroth term. Thus for $j = 0, 1$ we obtain a commutative diagram:
 \begin{center}
 \begin{tikzcd}
 {H^j(Y, \mc O_Y)} \arrow[r, "\cong" description] \arrow[d] & {H^j(Y, \pi_*^G \mc O_X)} \arrow[d] \\
 {H^j(Y, \Omega_{Y/k})} \arrow[r] & {H^j(Y, \pi_*^G \Omega_{X/k})},
 \end{tikzcd} 
 \end{center}
where the left arrow is zero and the upper arrow is an isomorphism.
Therefore for $j = 0, 1$ the maps
\[ H^j(Y, \pi_*^G \mc O_X) \to H^j(Y, \pi_*^G \Omega_{X/k}) \]
are zero. This is the desired conclusion.
\end{proof}
\begin{Corollary} \label{cor:delta}
	\begin{eqnarray*}
	\delta(X, G) &=& \bigg(\dim_k \HH^1(Y, \pi_*^G \Omega_{X/k}^{\bullet}) - \dim_k \HH^1(Y, \pi_* \Omega_{X/k}^{\bullet})^G \bigg) \\
	&-& \bigg( \dim_k H^1(Y, \pi_*^G \mc O_X) - \dim_k H^1(Y, \pi_* \mc O_X)^G \bigg).
	\end{eqnarray*}
\end{Corollary}
\begin{proof}
	By Lemma~\ref{lem:degeneration} we obtain an exact sequence:
	\[ 0 \to H^0(Y, \pi_*^G \Omega_{X/k}) \to \HH^1(Y, \pi_*^G \Omega_{X/k}^{\bullet}) \to H^1(Y, \pi_*^G \mc O_{X/k}) \to 0. \]
	Hence:
	\begin{align*}
	 \delta(X, G) &= \dim_k H^0(X, \Omega_{X/k})^G + \dim_k H^1(X, \mc O_X)^G 
	- \dim_k H^1_{dR}(X/k)^G \\
	&= \left( \dim_k \HH^1(Y, \pi_*^G \Omega_{X/k}^{\bullet}) - \dim_k H^1(Y, \pi_*^G \mc O_X) \right)\\
	&+ \dim_k H^1(X, \mc O_X)^G  - \dim_k H^1_{dR}(X/k)^G \\
	&= (\dim_k \HH^1(Y, \pi_*^G \Omega_{X/k}^{\bullet}) - \dim_k \HH^1(Y, \pi_* \Omega_{X/k}^{\bullet})^G) \\
	&- (\dim_k H^1(Y, \pi_*^G \mc O_X) - \dim_k H^1(Y, \pi_* \mc O_X)^G).
	\qedhere
	\end{align*}
\end{proof}
\noindent Corollary~\ref{cor:delta} implies that we need to compare the hypercohomology groups
\[ \HH^i(Y, (\mc F^{\bullet})^G) \textrm{ and } \HH^i(Y, \mc F^{\bullet})^G. \]
for $\mc F^{\bullet} = \pi_* \mc O_X$ (treated as a complex concentrated in degree $0$) and $\mc F^{\bullet} = \pi_* \Omega_{X/k}^{\bullet}$ (note that it is a complex of $\uu k_Y[G]$-modules
rather than $\mc O_Y$-modules, since the differentials in the de Rham
complex are not $\mc O_Y$-linear). Consider the 
commutative diagram:
\begin{center}
 \begin{tikzcd}
{\uu k_Y[G] \mmod} \arrow[r, "(-)^G"] \arrow[d, "{\Gamma(Y, -)}"] & \uu k_Y \mmod \arrow[d, "{\Gamma(Y, -)}"] \\
{k[G] \mmod} \arrow[r, "(-)^G"]                                 & k \mmod.                                
\end{tikzcd}
\end{center}
By applying Grothendieck spectral sequence to compositions of the functors in the diagram, we obtain two spectral sequences:
	\begin{eqnarray}
		_{I}E^{ij}_2 &=& \HH^i(Y, \mc H^j(G, \mc F^{\bullet})) \Rightarrow
		\RR^{i+j}\Gamma^G(\mc F^{\bullet}) \label{eqn:first_spectral_sequence}\\
		_{II} E^{ij}_2 &=& H^i(G, \HH^j(Y, \mc F^{\bullet})) \Rightarrow
		\RR^{i+j}\Gamma^G(\mc F^{\bullet}) \label{eqn:second_spectral_sequence},
	\end{eqnarray}
	(note that here $\mc H^j(G, \mc F^{\bullet})$ denotes a complex
	of $\uu k_Y$-modules with $l$th term being $\mc H^j(G, \mc F^l)$).\\
	For motivation, suppose at first that the 'obstructions'
	\[ \mc H^i(G, \mc F^l) \quad \textrm{ and } \quad H^i(G, \HH^l(Y, \mc F^{\bullet})) \]
	vanish for all $i \ge 1$ and $l \ge 0$ (this happens e.g. if $\cha k = 0$).
	Then the spectral sequences~\eqref{eqn:first_spectral_sequence} and~\eqref{eqn:second_spectral_sequence} lead us to the isomorphisms:
	\[
	  \HH^i(Y, (\mc F^{\bullet})^G) \cong 
	  \RR^{i}\Gamma^G(\mc F^{\bullet})
	  \cong (\HH^i(Y, \mc F^{\bullet}))^G.
	\]
	In general case the relation between $\HH^i(Y, (\mc F^{\bullet})^G)$  and $\HH^i(Y, \mc F^{\bullet})^G$ is more complicated. However, in the case of the first hypercohomology group, one can	extract some information from  the low-degree exact sequences of~\eqref{eqn:first_spectral_sequence} and \eqref{eqn:second_spectral_sequence}:
	\begin{eqnarray} 
		0 &\to& \HH^1(Y, (\mc F^{\bullet})^G) \to
		\RR^{1}\Gamma^G(\mc F^{\bullet}) \to \label{eqn:first_exact_sequence}\\
		&\to& \HH^0(Y, \mc H^1(G, \mc F^{\bullet})) \to \HH^2(Y, (\mc F^{\bullet})^G)
		\to \nonumber \\
		&\to& \RR^{2}\Gamma^G(\mc F^{\bullet}) \nonumber
	\end{eqnarray}
	and
	\begin{eqnarray}
		0 &\to& H^1(G, \HH^0(Y, \mc F^{\bullet})) \to
		\RR^{1}\Gamma^G(\mc F^{\bullet}) \to \label{eqn:second_exact_sequence}\\
		&\to& \HH^1(Y, \mc F^{\bullet})^G
		\to H^2(G, \HH^0(Y, \mc F^{\bullet})) \to \nonumber\\
		&\to& \RR^{2}\Gamma^G(\mc F^{\bullet}) \nonumber.
	\end{eqnarray}
	This will be done separately in the case of wild and tame ramification in the
	Subsections~\ref{subsec:wild_case} and~\ref{subsec:tame_case}.
\subsection{Proof -- the wild case} \label{subsec:wild_case} 
Consider first the case when $\pi$ is wildly ramified, i.e.
by Lemma~\ref{lem:vanishing_of_r_prim} when $R' \neq 0$. Then, as one easily sees by Lemma~\ref{lem:degeneration}, Corollary~\ref{cor:pi_G_omega} and Riemann--Roch theorem (cf. \cite[Theorem IV.1.3]{Hartshorne1977}):
\begin{eqnarray*}
	\HH^2(Y, \pi_*^G \Omega_{X/k}^{\bullet}) &=& H^1(Y, \pi_*^G \Omega_{X/k})\\
	&=& H^1(Y, \Omega_{Y/k} \otimes \mc O_Y(R'))\\
	&=& 0.
\end{eqnarray*}
By~\eqref{eqn:first_exact_sequence} we see that 
		\begin{eqnarray}
			\dim_k \RR^1 \Gamma^G (\pi_* \Omega_{X/k}^{\bullet}) &=& \dim_k \HH^1(Y, (\pi_* \Omega_{X/k}^{\bullet})^G) \label{eqn:R1_first} \\
			&+& \dim_k \HH^0(Y, \mc H^1(G, \pi_* \Omega_{X/k}^{\bullet})). \nonumber
		\end{eqnarray} 
On the other hand, \eqref{eqn:second_exact_sequence} yields:
		\begin{eqnarray}
			\dim_k \RR^1 \Gamma^G (\pi_* \Omega_{X/k}^{\bullet}) &=& \dim_k H^1(G, \HH^0(Y, \pi_* \Omega_{X/k}^{\bullet})) \label{eqn:R1_second}\\
			&+& \dim_k \HH^1(Y, \pi_* \Omega_{X/k}^{\bullet})^G - c_1, \nonumber
		\end{eqnarray} 
		where 
		\begin{eqnarray}
			c_1 &=& \dim_k \ker \left( H^2(G, \HH^0(Y, \pi_* \Omega_{X/k}^{\bullet})) \to \RR^2 \Gamma^G (\pi_* \Omega_{X/k}^{\bullet}) \right). \label{eqn:c1}
		\end{eqnarray}
		Thus by comparing~\eqref{eqn:R1_first} and \eqref{eqn:R1_second}:
		\begin{eqnarray}
			\dim_k \HH^1(Y, \pi_* \Omega_{X/k}^{\bullet})^G &=& \dim_k \HH^1(Y, (\pi_* \Omega_{X/k}^{\bullet})^G) \label{eqn:invariants_of_hypercohomology} \\
			&+& \dim_k \HH^0(Y, \mc H^1(G, \pi_* \Omega_{X/k}^{\bullet})) \nonumber \\ 
			&-& \dim_k H^1(G, \HH^0(Y, \pi_* \Omega_{X/k}^{\bullet})) + c_1 \nonumber
		\end{eqnarray}
		By repeating the same argument for
		$\pi_* \mc O_X$, we obtain:
		\begin{eqnarray}
			\dim_k H^1(Y, \pi_* \mc O_X)^G &=& \dim_k H^1(Y, (\pi_* \mc O_X)^G) \label{eqn:invariants_of_cohomology} \\
			&+& \dim_k H^0(Y, \mc H^1(G, \pi_* \mc O_X)) \nonumber \\ 
			&-& \dim_k H^1(G, H^0(Y, \pi_* \mc O_X)) + c_2, \nonumber
		\end{eqnarray} 
		where:
		\begin{eqnarray}
			c_2 &=& \dim_k \ker \left( H^2(G, H^0(Y, \pi_* \mc O_X)) \to 
			R^2 \Gamma^G (\pi_* \mc O_X)) \right). \label{eqn:c2}
		\end{eqnarray}
		By combining~\eqref{eqn:invariants_of_hypercohomology}, \eqref{eqn:invariants_of_cohomology} and Corollary~\ref{cor:delta} we obtain:
		\begin{eqnarray*}
		 \delta(X, G) &=& \dim_k \im\bigg( H^0(Y, \mc H^1(G, \pi_* \mc O_X))
		 \to H^0(Y, \mc H^1(G, \pi_* \Omega_{X/k})) \bigg) \\
		 &+& (c_2 - c_1).
		\end{eqnarray*}
		Note that since $\mc H^1(G, \pi_* \mc O_X)$, $\mc H^1(G, \pi_* \Omega_{X/k})$ are torsion sheaves, we can compute their sections by taking stalks and using~\eqref{eqn:group_cohomology_stalk}:
		\begin{eqnarray*}
		 \dim_k \im\bigg( H^0(Y, \mc H^1(G, \pi_* \mc O_X))
		 \to H^0(Y, \mc H^1(G, \pi_* \Omega_{X/k})) \bigg) = \\
		 \sum_{Q \in Y} \dim_k \im \bigg(H^1(G, (\pi_* \mc O_X)_Q)
		 \to H^1(G, (\pi_* \Omega_{X/k})_Q) \bigg)
		\end{eqnarray*}
		Thus we are left with showing
		that $c_1 = c_2$. This will be done in Subsection~\ref{subsec:c1=c2}.
		
\subsection{Proof -- the tame case} \label{subsec:tame_case}

Consider now the case of tame ramification, i.e. $R' = 0$. Then by Propostion~\ref{prop:gr_coh_torsion_sheaf} we see that $\mc H^i(G, \pi_* \Omega_{X/k}^j) = 0$
for $i \ge 1$, $j \ge 0$. 
Thus it is evident by~\eqref{eqn:first_spectral_sequence} that
\[ \RR^i \Gamma^G (\pi_* \Omega_{X/k}^{\bullet}) \cong
\HH^i(Y, (\pi_* \Omega_{X/k}^{\bullet})^G). \]
Therefore the exact sequence~\eqref{eqn:second_exact_sequence} implies that:
\begin{eqnarray*}
 \dim_k \HH^1(Y, \pi_* \Omega_{X/k}^{\bullet})^G &=&
 \dim_k \HH^1(Y, \pi_*^G \Omega_{X/k}^{\bullet}) + \dim_k H^1(G, k) + c_1,
\end{eqnarray*}
where $c_1$ is given by~\eqref{eqn:c1}. One proceeds analogously as in the wildly ramified
case to obtain:
		\begin{eqnarray*}
		 \delta(X, G) = (c_2 - c_1).
		\end{eqnarray*}
Again, it remains to prove that $c_1 = c_2$.

\subsection{Proof -- the end} \label{subsec:c1=c2}
Recall that in order to prove Proposition~\ref{prop:main_lemma} we have to
investigate the map
\begin{equation} \label{eqn:map}
 H^2(G, \HH^0(Y, \mc F^{\bullet})) \to \RR^2 \Gamma^G (\mc F^{\bullet})
\end{equation}
arising from the exact sequence~\eqref{eqn:second_exact_sequence}.

\begin{Lemma} \label{lem:monomorphism_R^2}
Let $\mc F^{\bullet}$ be complex of $\mc O[G]$-sheaves on a ringed 
space $(Y, \mc O)$, which is a noetherian topological space of dimension $1$. Suppose that $\mc F^i = 0$ for $i \neq 0, 1$ and
that the support of the sheaf $\mc H^i(G, \mc F^j)$ is a finite subset of~$Y$ for $i \ge 1$.
There exists a natural monomorphism
\[
	\HH^0(Y, \mc H^2(G, \mc F^{\bullet})) \hookrightarrow \RR^2 \Gamma^G(\mc F^{\bullet}).
\]
It is an isomorphism, provided that $\mc F^{\bullet}$ is a complex concentrated in
degree~$0$. 
\end{Lemma}
\begin{proof}
Note that for $i \ge 2$, $j \ge 1$
\[
	_{I}E^{ij}_2 = \HH^i(Y, \mc H^j(G, \mc F^{\bullet})) = 0.
\]
Indeed, this follows by~\eqref{eqn:spectral_sequence_hypercohomology}, since for every $l$, $H^i(Y, \mc H^j(G, \mc F^l)) = 0$ for $i, j \ge 1$ (cf. \cite[Theorem III.2.7]{Hartshorne1977}), for $i \ge 2$ and for $j \ge 2$. Thus it is evident that there exists a natural monomorphism
\[
	\HH^0(Y, \mc H^2(G, \mc F^{\bullet})) = {}_IE^{02}_2
 	= {}_IE^{02}_{\infty} \hookrightarrow \RR^2 \Gamma^G (\mc F^{\bullet}).
\]
Suppose now that $\mc F^{\bullet}$ is concentrated in degree $0$.
Then $_IE^{ij}_2 = 0$ for $i, j \ge 1$ and for $i \ge 2$.
Therefore ${}_IE^{11}_{\infty} = {}_IE^{11}_2 = 0$ and
$_IE^{20}_{\infty} = {}_IE^{20}_2 = 0$, which leads to the conclusion. \qedhere

\end{proof}
\begin{Corollary} \label{cor:commutative_diagram}
 There exists a commutative diagram
			      \begin{center}
                    \begin{tikzcd}
                    {H^2(G, \HH^0(Y, \mc F^{\bullet}))} \arrow[r] \arrow[d] & \RR^2 \Gamma^G (\mc F^{\bullet}) \\
                    {\HH^0(Y, \mc H^2(G, \mc F^{\bullet}))} \arrow[ru, hook] & 
                    \end{tikzcd}
			      \end{center}
			      where the upper arrow is~\eqref{eqn:map},
			      and the lower arrow is as in Lemma~\ref{lem:monomorphism_R^2}.
\end{Corollary}
\begin{proof}
	The morphism $\mc F^{\bullet} \to \mc F^0$ (where we treat $\mc F^0$ as a complex concentrated
	in degree $0$) yields by functoriality
	the commutative diagram:
	\begin{center}
	 	\begin{tikzcd}
		{H^2(G, \HH^0(Y, \mc F^{\bullet}))} \arrow[d] \arrow[r] & \RR^2 \Gamma^G (\mc F^{\bullet}) \arrow[d] & {\HH^0(Y, \mc H^2(G, \mc F^{\bullet}))} \arrow[d] \arrow[l, hook'] \\
		{H^2(G, H^0(Y, \mc F^0))} \arrow[r] & R^2 \Gamma^G(\mc F^0) & {H^0(Y, \mc H^2(G, \mc F^0))} \arrow[l, "\cong"'].
		\end{tikzcd}
	\end{center}
	By composing the maps from the diagram we obtain a map
	\begin{eqnarray} 
	 H^2(G, \HH^0(Y, \mc F^{\bullet})) &\rightarrow& H^2(G, H^0(Y, \mc F^0)) 
	\label{eqn:composition_of_maps}\\
	 \rightarrow R^2 \Gamma^G (\mc F^0) &\cong& H^0(Y, \mc H^2(G, \mc F^0)). \nonumber 
	\end{eqnarray}
	One easily checks that the image of the map~\eqref{eqn:composition_of_maps} 
	lies in the image of 
	\[ \HH^0(Y, \mc H^2(G, \mc F^{\bullet})) \hookrightarrow H^0(Y, \mc H^2(G, \mc F^0)).\]
	This clearly completes the proof.
\end{proof}
We are now ready to finish the proof of Proposition~\ref{prop:main_lemma}.
Recall that we are left with showing that $c_1 = c_2$ (where $c_1$ and
$c_2$ are given by~\eqref{eqn:c1} and~\eqref{eqn:c2} respectively).
By using Corollary~\ref{cor:commutative_diagram} for $\mc F^{\bullet} = \pi_* \Omega_{X/k}^{\bullet}$, Lemma~\ref{lem:monomorphism_R^2}
and the equality
\[ \HH^0(Y, \pi_* \Omega_{X/k}^{\bullet}) = H^0(Y, \pi_* \mc O_X) = k \]
we obtain:
\begin{eqnarray*}
 c_1 &=& \dim_k \ker \left( H^2(G, \HH^0(Y, \pi_* \Omega_{X/k}^{\bullet})) \to \RR^2 \Gamma^G (\pi_* \Omega_{X/k}^{\bullet}) \right)\\
 &=& \dim_k \ker \left( H^2(G, \HH^0(Y, \pi_* \Omega_{X/k}^{\bullet})) \to
 \HH^0(Y, \mc H^2(G, \pi_* \Omega_{X/k}^{\bullet})))) \right)\\
 &=& \dim_k \ker \left( H^2(G, \HH^0(Y, \pi_* \mc O_X)) \to H^0(Y, \mc H^2(G, \pi_* \mc O_X))) \right)\\
 &=& \dim_k \ker \left( H^2(G, \HH^0(Y, \pi_* \mc O_X)) \to R^2 \Gamma^G(\pi_* \mc O_X))) \right)\\
 &=& c_2.
\end{eqnarray*}

\section{Computation of local terms} \label{sec:local_computations}

\subsection{Proofs of main results}
The main goal of this section is to compute the local terms occuring in
Proposition~\ref{prop:main_lemma}. This is achieved in the following
proposition.
\begin{Proposition} \label{prop:ker_na_kohomologii}
	Keep the notation introduced in Section~\ref{sec:intro} and suppose that $G \cong \ZZ/p$.
	Then for any $Q \in Y$ the dimension of 
	\begin{eqnarray*}
	\im \bigg(H^1(G, (\pi_* \mc O_X)_Q) \to H^1(G, (\pi_* \Omega_{X/k})_Q) \bigg)
	\end{eqnarray*}
	equals
	\begin{equation*}
	\left[\frac{(n_Q + 1) \cdot (p-1)}{p}\right] - 1 - \left[\frac{n_Q - 1}{p}\right]. 
	\end{equation*}
\end{Proposition}
Proposition~\ref{prop:ker_na_kohomologii} will be proven in the
Subsection~\ref{subsec:proof_of_prop}. We now show how the
Proposition~\ref{prop:ker_na_kohomologii} implies
the Theorems announced in the Introduction.
		\begin{proof}[Proof of Theorem~\ref{thm:invariants_AS}]
			Theorem~\ref{thm:invariants_AS} follows immediately by combining
			Propositions~\ref{prop:main_lemma} and~\ref{prop:ker_na_kohomologii}.
		\end{proof}
		\begin{proof}[Proof of Main Theorem]
			We consider first the case $G = \ZZ/p$. 
			An easy computation shows that for any $n \ge 1$, $p \ge 3$
			one has:
			\[
			\left[\frac{(p-1) \cdot (n + 1)}{p}\right] \ge 1 + \left[\frac{n - 1}{p}\right]
			\]
			with an equality only for $n = 1$ (here is where we use the assumption $p > 2$). Thus by Theorem~\ref{thm:invariants_AS},
			$\delta(X, G) = 0$ holds if and only if $\pi$ is weakly ramified.\\

		Suppose now that $G$ is arbitrary and $G_{P, 2} \neq 0$ for some $P \in X$.
		Note that $G_{P, 2}$ is a $p$-group (cf. \cite[Corollary 4.2.3., p. 67]{Serre1979}) and thus contains a subgroup $H$ of order~$p$.
		Observe that $\pi : X \to X/H$ is an Artin-Schreier covering and
		it is non-weakly ramified, since $H_{P, 2} = H \neq 0$. Therefore by the first paragraph of the proof, the sequence~\eqref{eqn:hodge_de_rham_se} does not split $H$-equivariantly
		and therefore it cannot split as a sequence of $k[G]$-modules.
		\end{proof}

\subsection{Galois cohomology of sheaves on Artin-Schreier coverings} \label{subsec:proof_of_prop}
We start by recalling the most important facts concerning Artin--Schreier coverings.
For a reference see e.g.~\cite[sec. 2.2]{PriesZhu_p_rank_AS}.
Let $X$ be a smooth algebraic curve with an action of $G = \ZZ/p$ over an algebraically closed
field~$k$ of characteristic~$p$. By Artin--Schreier theory, the function field of $X$ is given by
the equation
\begin{equation} \label{eqn:artin_schreier}
	z^p - z = f
\end{equation}
for some $f \in k(Y)$, where $Y := X/G$. The action of $G = \langle \sigma \rangle \cong \ZZ/p$ is 
then given by $\sigma(z) := z+1$. Let $\mc P \subset Y$ denote the set of points
at which $\pi$ is ramified. Note that $\mc P$ is contained in the set of poles
of~$f$ and moreover for any $Q \in Y$:
\begin{eqnarray*}
	\# \pi^{-1}(Q) =
	\begin{cases}
		p, & \textrm{ for } Q \not \in \mc P, \\
		1, & \textrm{ otherwise. }            
	\end{cases}
\end{eqnarray*}
\begin{Lemma} \label{lem:nP_of_AS}
 Keep the above setting. Fix a point $Q \in \mc P$ and let $\pi^{-1}(Q) = \{ P \}$.
 Suppose that $p \nmid n:= \ord_Q(f)$. Then for some $t \in \wh{\mc O}_{X, P}$
 and $x \in \wh{\mc O}_{Y, Q}$:
 \begin{itemize}
  \item $\wh{\mc O}_{X, P} = k[[t]]$, $\wh{\mc O}_{Y, Q} = k[[x]]$,
  \item $t^{-np} - t^{-n} = x^{-n}$,
  \item the action of $G \cong \ZZ/p$ on $t$ is given by an automorphism:
\begin{equation} \label{eqn:sigma_t}
	\sigma(t) = \frac{t}{(1 + t^n)^{1/n}} = t - \frac{1}{n} t^{n + 1} + 
	\text{ (terms of order $\ge n+2$)}.
\end{equation} 
 \end{itemize}
In particular, $n$ is equal to $n_Q$ as defined by~\eqref{eqn:definition_of_nP}.
\end{Lemma}
\begin{proof}
 Let $x$, $t$ be arbitrary uniformizers at $Q$ and $P$ respectively.
 Then $\wh{\mc O}_{Y, Q} = k[[x]]$ and $\wh{\mc O}_{X, P} = k[[t]]$.
 Before the proof observe that an equation $u^m = h(x)$ has a solution
 $u \in k[[x]]$, whenever $p \nmid m$ and $m | \ord(h)$ (this follows easily from
 Hensel's lemma). We will denote any solution by $h(x)^{1/m}$.\\
 Note that:
\[ f^{-1} = \frac{z^{-p}}{1 - z^{1-p}}. \]
By comparing the valuations we see that $\ord_P(z) = -n$.
Thus we may replace $t$ by $z^{-1/n}$ to ensure that $z = t^{-n}$. Then:
\begin{eqnarray*}
 \sigma(t)^n &=& \sigma(t^n) = \sigma \left(\frac{1}{z} \right)\\
 &=& \frac{1}{z + 1} = \frac{1}{t^{-n} + 1} = \frac{t^n}{1 + t^n}
\end{eqnarray*}
and thus we can assume without loss of generality (by replacing $\sigma$ by its power if necessary) that $\sigma(t) = \frac{t}{(1 + t^n)^{1/n}}$. Finally, we replace
$x$ by $f(x)^{-1/n}$ to ensure that $t^{-np} - t^{-n} = x^{-n}$.
\end{proof}
\begin{Example}\label{ex:hyperelliptic2}
 Let $X$ be the curve considered in Example~\ref{ex:hyperelliptic1}. Then $X$ is a $\ZZ/p$-covering of a curve $Y$ with the affine equation:
 \begin{equation*}
  y^m = f(x).
 \end{equation*}
 The function field of $X$ is given by the equation $z^p - z = x$.
 As proven in~\cite{Towse1996} the function $x \in k(Y)$ has $\delta := \GCD(m, \deg f)$
 poles, each of them of order $m/\delta$. This establishes the formula~\eqref{eqn:nP}.
\end{Example}

\begin{Remark}
 Suppose that $\pi : X \to Y$ is an Artin-Schreier covering. For every point $Q \in \mc P$ we can find functions $f_Q \in k(Y)$, $z_Q \in k(X)$ such that the function field of $X$ is given by the
 equation $z_Q^p - z_Q = f_Q$ and either $f_Q \in \mc O_{Y, Q}$, or $p \nmid \ord_Q(f_Q)$. 
 Indeed, in order to obtain $f_Q$ one can repeatedly subtract from $f$ a function
 of the form $h^p - h$, where $h$ is a power of a uniformizer at $Q$.
\end{Remark}
\begin{Example}
 It might not be possible to find a function $f$ such the function field of
 $X$ is given by~\eqref{eqn:artin_schreier} and for any pole $Q$ of $f$
 one has $p \nmid  \ord_Q(f)$. Take for example an ordinary elliptic curve $X / \ol{\FF}_p$.
 Let $\tau \in \Aut(X)$ be a translation by a $p$-torsion point.
 Consider the action of $G = \langle \tau \rangle \cong \ZZ/p$ on~$X$.
 This group action is free and hence $n_P = 0$ for all $P \in X$. Thus, if $k(X)$ would have
 an equation of the form $z^p - z = f$, where $p \nmid \ord_Q(f)$ for all
 $Q \in \mc P$, then $f$ would have no poles. This easily leads
 to a contradiction.
\end{Example}
\noindent Keep the notation of Lemma~\ref{lem:nP_of_AS}.
Fix an integer $a \in \ZZ$ and denote:
	\begin{itemize}
		\item $B := \wh{\mc O}_{Y, Q} = k[[t]]$, $L := k((t))$, $I := t^a B$,
		\item $A := \wh{\mc O}_{X, P} = k[[x]]$, $K := k((x))$.
	\end{itemize}	
In the below Lemma we will compute $H^1(G, I)$.	The dimension of $H^1(G, I)$ is computed also in \cite[Th\'{e}or\'{e}me 4.1.1]{BertinMezard2000} (see also \cite[formula (18)]{Kontogeorgis}). However, we need an explicit description of
a basis of $H^1(G, I)$.
\begin{Lemma} \label{lem:H^1(I)}
\begin{enumerate}[(1)]
 \item $H^1(G, I)$ may be identified with \[ M := \coker(L^G \to (L/I)^G). \]
 \item A basis of $H^1(G, I)$ is given by the images of the elements $(t^i)_{i \in J}$ in $M$, where
	\[ J := \{ a - n \le i \le a - 1 : p \nmid i \}.  \]
 \item $\dim_k H^1(G, I) = n - \left[\frac{a - 1}{p} \right] + \left[\frac{a - 1 - n}{p} \right]$.
 \item The images of the elements:
	\[ t^i \quad \textrm{ for } a - n \le i \le a - 1, \quad p | i \]
	are zero in $M$.
\end{enumerate}

		\end{Lemma}
		\begin{proof}
		For any $h \in L$, we will denote its images in $L/I$ and $M$ by $[h]_{L/I}$
		and $[h]_M$ respectively.
		
		\begin{enumerate}[(1)]
		 \item The proof follows by taking the long exact sequence of cohomology for the short exact sequence of $k[G]$-modules:
			\[ 0 \to I \to L \to L/I \to 0 \]
			and using the Normal Base Theorem (cf. \cite[sec. 4.14]{Jacobson_basic_I}).
			
			\item Note that for any $a-n \le i \le a-1$, we have $[t^i]_{L/I} \in (L/I)^G$, since
			\begin{eqnarray*}
				\sigma([t^i]_{L/I}) &=& [\sigma(t^i)]_{L/I} = [(t - \frac{1}{n}t^{n+1} + O(t^{2n}))^i]_{L/I} =\\
				&=& [t^i - \frac{i}{n} t^{i + n} + O(t^{2n})]_{L/I} = [t^i]_{L/I}.
			\end{eqnarray*}
			We'll show now that the set $([t^i]_{M})_{i \in J}$ spans $M$.
			Note that $L^G = K$. Therefore it suffices to show that for any $[h]_{L/I} \in (L/I)^G$, one has
			\begin{equation} \label{eqn:h_in_A}
			 h \in K + \bigoplus_{i \in J} k \cdot t^i.
			\end{equation}
			Let $h = \sum_{i = N}^{a - 1} a_i t^i \in L$, where $a_N \neq 0$. Observe that if $p|j$ and $a_j \neq 0$, then we may replace~$h$ by $h - c \cdot x^{j/p}$  for a suitable constant $c \in k$,
			since valuation of $x$ in~$L$ equals~$p$. Thus without loss of generality we may assume that
			$a_j = 0$ for $p|j$ and that $p \nmid N$. The equality $\sigma([h]_{L/I}) = [h]_{L/I}$ is equivalent to
			\begin{equation*}
				\sum_{i = N}^{a - 1} a_i \sigma(t)^i = \sum_{i = N}^{a - 1} a_i t^i + \sum_{i = a}^{\infty} b_i t^i
			\end{equation*}
			for some $b_a, b_{a + 1}, \ldots \in k$. By using equality~\eqref{eqn:sigma_t} this implies:
			\begin{equation*}
				\sum_{i = N}^{a - 1} a_i t^i \cdot \left(1 - \frac{i}{n}t^n + O(t^{2n})\right) = \sum_{i = N}^{a - 1} a_i t^i + \sum_{i = a}^{\infty} b_i t^i.
			\end{equation*}
			By comparing coefficients of $t^{N + n}$, we see that either $N+n \ge a$,
			or
			\[ 
				a_N \cdot \left(-\frac{N}{n} \right) + a_{N+n} = a_{N+n}.
			\]
			The second possibility easily leads to a contradiction. This proves~\eqref{eqn:h_in_A}. 
			We check now linear independence of the considered elements. Suppose that for some $a_i \in k$
			not all equal to zero:
			\[ \sum_{i \in J} a_i [t^i]_M = 0 \]
			or equivalently,
			\begin{equation} \label{eqn:linear_dependence}
			 \sum_{i \in J} a_i t^i = \sum_{j \ge N} b_j x^j + \sum_{j \ge a} c_j t^j 
			\end{equation}
			for some $b_j, c_j \in k$, $b_N \neq 0$.
			Consider the coefficient of $t^{pN}$ in~\eqref{eqn:linear_dependence}.
			Observe that $x = t^p + O(t^{p+1})$, since $\ord_P(x) = p$. We see that
			either $pN \ge a$ (which is impossible, since $\sum_{i \in J} a_i t^i \not \in I$)
			or $0 = b_N + 0$, which also leads to a contradiction. This ends the proof.
			
			\item Follows immediately by (2).
			
			\item Note that
			\begin{eqnarray*}
			 x &=& \frac{1}{(t^{-np} - t^{-n})^{1/n}} = \frac{t^p}{(1 - t^{n \cdot (p - 1)})^{1/n}}\\
			 &=& t^p \cdot (1 + O(t^{n \cdot (p-1)}))
			\end{eqnarray*}
			and thus for any $a - n \le i \le a - 1$, $p | i$:
			\[ x^{i/p} = t^i \cdot (1 + O(t^{n \cdot (p-1)})) = t^i + O(t^a), \]
			and $[t^i]_{L/I} = [x^{i/p}]_{L/I}$, which shows that $[t^i]_M = 0$. \qedhere
		\end{enumerate}
		\end{proof}
		
		\begin{proof}[Proof of Proposition~\ref{prop:ker_na_kohomologii}:]
			Fix a point $Q \in \mc P$ and keep the above notation.
			Note that $\wh{(\pi_* \mc O_X)_Q} \cong B$, $\wh{\pi_*\Omega_{X/k}} = B \, dt$.
			Moreover, note that $\frac{dt}{t^{n+1}}$ is a $G$-invariant form, since
			from the equation $t^{-np} - t^{-n} = x^{-n}$ one obtains:
			\begin{eqnarray*}
				\frac{dt}{t^{n+1}} = - \frac{dx}{x^{n+1}}.
			\end{eqnarray*}
			Thus we have the following isomorphism of $B[G]$-modules:
			\begin{eqnarray*}
				B \, dt &\longrightarrow& t^{n+1} B\\
				h(t) \, dt = t^{n+1} h(t) \cdot \frac{dt}{t^{n+1}} &\longmapsto& t^{n+1} h(t).
			\end{eqnarray*}
			(cf. \cite[proof of Lemma 1.11.]{Kontogeorgis} for the "dual" version of this isomorphism). 
			Lemma~\ref{lem:H^1(I)} implies that
			$H^1(G, B)$ and $H^1(G, B \, dt)$ may be identified with
			\begin{equation*}
			 M_1 := \coker(L^G \to (L/B)^G) \quad \textrm{ and } \quad
			 M_2 := \coker((L \, dt)^G \to (L \, dt/B \, dt)^G )
			\end{equation*}
			respectively. One easily checks that the morphism $d : H^1(G, B) \to H^1(G, B \, dt)$
			corresponds to
			\[ d : M_1 \to M_2, \quad d([h(t)]_{M_1}) = [dh(t)]_{M_2} = [h'(t) \, dt]_{M_2}. \]
			By using Lemma~\ref{lem:H^1(I)} (2), (4) for $a = 0$ and $a = n+1$ we
			see that the basis of $\im(d : M_1 \to M_2)$
			is 
			\begin{equation*}
				[dt^i]_{M_2} = [i t^{i-1} \, dt]_{M_2} \quad \textrm{ for } i = -n, -n + 1, \ldots, -1, \quad
				p \nmid i, \quad i+n \not \equiv 0 \pmod p.
			\end{equation*}
			An elementary calculation allows one to compute the dimension of this space.
		\end{proof}

\section{Converse results} \label{sec:converse}
This section will be devoted to proving
various converse statements to Main Theorem.
		\subsection{A counterexample} \label{subsec:counterexample}
		We start this section by giving an example of an elliptic curve
		over a field of characteristic $2$ with a weakly ramified group action, for which the
		sequence~\eqref{eqn:hodge_de_rham_se} does not split equivariantly. It remains unclear whether similar counterexamples will arise over fields of different characteristic.
		
		Consider the elliptic curve $X$ over the field $k := \ol{\FF}_2$ with the affine
		part $U_0$ given by the equation:
		\[ y^2 + y = x^3. \]
		Note that $X \setminus U_0 = \{ \mc O \}$, where $\mc O$ is the
		point at infinity. The group $G$ of automorphisms of $X$ that fix $\mc O$ is of order $24$ and is
		isomorphic to $\Sl_2(\FF_3)$. In particular its $2$-Sylow subgroup is isomorphic to the quaternion group~$Q_8$. This group action may be given explicitly, cf.~\cite[Appendix A]{Silverman2009} or~\cite[Section 3]{KronbergSoomroTop}.
		Let:
		\[ A := \{ (u, r, t) : u \in \FF_4^{\times}, \quad t \in \FF_4, \quad t^2 + t + r^3 = 0 \}. \]
		Define for any $(u, r, t) \in A$ an automorphism $g_{u, r, t} \in \Aut(X)$ by:
		\[ g_{u, r, t} \cdot (x, y) := (u^2 x + r, y + u^2 r^2 x + t). \] 
		We'll compute $H^1_{dR}(X/k)$ using \v{C}ech cohomology.
		Recall that if a curve $X$ may be covered by affine subsets $U_0, U_{\infty}$,
		then:
		\begin{eqnarray*}
		 H^1_{dR}(X/k) &\cong& \frac{\{ (\omega_0, \omega_{\infty}, f_{0 \infty}) : df_{0 \infty} = \omega_0 - \omega_{\infty} \}}
		 {\{ (df_0, df_{\infty}, f_0 - f_{\infty}): f_i \in \mc O_X(U_i) \}},
		\end{eqnarray*}
		where we take $\omega_i \in \Omega_X(U_i)$ for $i = 0, \infty$ and $f_{0 \infty} \in \mc O_X(U_0 \cap U_{\infty})$.
		In our case, we may take $U_0$ as above and $U_{\infty} = X \cap \{ x \neq 0 \}$.
		Then, by~\cite[Theorem 2.2.]{KockTait2018}, one sees that $H^1_{dR}(X/k)$ is a $k$-vector space
		of dimension~$2$, generated by $v_1 := [(dx, dx, 0)]$ and $v_2 := [(x \, dx, \frac{y \, dx}{x^2}, \frac{y}{x})]$.
		\begin{Lemma}
		 In the above situation:
		 \begin{enumerate}[(a)]
		  \item $H^1_{dR}(X/k)$ is an indecomposable $k[G]$-module,
		  \item the action of $G$ on $X$ is weakly ramified.
		 \end{enumerate}
		\end{Lemma}
		\begin{proof}
		\begin{enumerate}
		 \item Suppose that $V$ is a $G$-invariant proper subspace
		 of $H^1_{dR}(X/k)$. We'll show that $V = \Span_k(v_1)$. Indeed, otherwise
		 we would have $V = \Span_k(v)$ for some $v = \alpha \cdot v_1 + v_2$ and $\alpha \in k$.
		 Note that for $g = g_{u, r, t}$:
		 \begin{eqnarray*}
		  g v_1 &=& u^2 v_1\\
		  g v_2 &=& u^2 t v_1 + u v_2.
		 \end{eqnarray*}
		Thus:
		\begin{eqnarray*}
		 g \cdot v = (\alpha \cdot u^2 + u^2 \cdot t) \cdot v_1 + u v_2.
		\end{eqnarray*}
		Therefore $g \cdot v \in V$ if and only if $\alpha \cdot u^2 + u^2 \cdot t = u \alpha$,
		which leads to the equation:
		\[ (1 - u) \cdot \alpha = ut. \]
		The last equality is however impossible to hold for all $(u, r, t) \in A$
		and a fixed $\alpha \in k$. Indeed, one can take e.g. $(u, r, t) = (1, 0, 1), (1, 1, \zeta)$
		for any $\zeta \in \FF_4 \setminus \FF_2$ to obtain a desired contradition.
		
		\item One easily sees that if $g P = P$ and $g \neq \id$ then
		$P = \mc O$. Thus we are left with showing that $G_{\mc O, 2} = 0$.
		Observe that $\ord_{\mc O}(x) = -2$ and $\ord_{\mc O}(y) = -3$.
		Hence the function $t := \frac{x}{y}$ is the uniformizer at $\mc O$.
		For $g = g_{u, r, t}$ one has:
		\begin{eqnarray*}
		 g(t) - t = \frac{(u^2 + 1) c \cdot xy + u^2 r^2 \cdot x^2 + r \cdot y + t \cdot x}
		 {y \cdot (y + u^2 r^2 x + t)}
		\end{eqnarray*}
		and
		\begin{equation*}
		 \ord_{\mc O}(g(t) - t) =
		 \begin{cases}
		  2, & \textrm{ if } u = 1, \\
		  1, & \textrm{ if } u \neq 1.\\
		 \end{cases}
		\end{equation*}
		Therefore $G_{\mc O, 2} = 0$ and
		\[ G_{\mc O, 1} = \{ g_{1,r,s} : (1,r,s) \in A \} \cong Q_8. \]
		
		\end{enumerate}
		\end{proof}

\subsection{The $G$-fixed subspaces}
The methods used throughout the article seem to be insufficient 
to obtain a positive result regarding splitting of the exact
sequence~\eqref{eqn:hodge_de_rham_se}. However, we may say something about the $G$-fixed
subspaces in the sequence~\eqref{eqn:hodge_de_rham_se}.
\begin{proof}[Proof of Theorem~\ref{thm:weakly_then_invariants}]
			By Proposition~\ref{prop:main_lemma} it is sufficient to
			show that the map 
			\[ H^1(G, (\pi_* \mc O_X)_{\pi(P)}) \to H^1(G, (\pi_* \Omega_{X/k})_{\pi(P)}) \]
			is zero for every $P \in X$. Just as in the proof of Proposition~\ref{prop:invariants}
			we observe that 
			\begin{eqnarray*}
				H^1(G, (\pi_* \mc O_X)_{\pi(P)}) \cong H^1(G_P, k) \oplus
				H^1(G_P, \mf m_{X, P}).
			\end{eqnarray*}
			However, the map $d : k \to \Omega_{X/k}$ is zero
			and thus the induced map 
			\[ d : H^1(G_P, k) \to H^1(G, (\pi_* \Omega_{X/k})_{\pi(P)}) \]
			is also zero. Moreover, since 
			$\pi$ is weakly ramified, by a result of K\"{o}ck (cf. \cite[Theorem 1.1]{Kock2004}), $H^1(G_P, \mf m_{X, P}) = 0$.
			This ends the proof.
		\end{proof}
		Note that if an action of a finite group $G$ on $X$ is weakly ramified
		then the action of any subgroup of $G$ on $X$ is also weakly ramified.
		Therefore the condition imposed by Theorem~\ref{thm:weakly_then_invariants} on the Hodge--de Rham exact sequence of
		$X$ seems to be strong from the group theoretical point of view.
		This raises the following question:
		\begin{Question} \label{q:splitting_k[G]-mod}
		 Suppose that $k$ is a field of characteristic $p > 0$ and $G$
		 is a finite group. Let also
		 \begin{equation} \label{eqn:ABC_exact_sequence}
		  0 \to A \to B \to C \to 0
		 \end{equation}
		 be an exact sequence of $k[G]$-modules of finite dimension over $k$.
		 Assume that for every subgroup $H \le G$ the sequence
		 \[ 0 \to A^H \to B^H \to C^H \to 0 \]
		 is exact also on the right. Does it follow that the exact
		 sequence~\eqref{eqn:ABC_exact_sequence} splits $G$-equivariantly?
		\end{Question}
		The results of the Subsection~\ref{subsec:counterexample} show that the answer to the Question~\ref{q:splitting_k[G]-mod}
		is negative for $\cha k = 2$. The following lemma reduces the Question \ref{q:splitting_k[G]-mod} to the case of $p$-groups.
		\begin{Lemma} \label{lem:splitting_sylow}
			Let $k$ be a field of characteristic $p > 0$
			and let $G$ be a finite group with a $p$-Sylow subgroup $P$.
			Suppose that
			\begin{equation} \label{eqn:kP-modules}
				0 \to A \to B \to C \to 0
			\end{equation}			
			is an exact sequence of $k[G]$-modules. Then \eqref{eqn:kP-modules}
			splits as an exact sequence of $k[G]$-modules
			if and only if it splits as an exact sequence of $k[P]$-modules.
		\end{Lemma}
		\begin{proof}
			The proof is adapted from the proof of Maschke theorem.
			Suppose that $s : C \to B$ is a $k[P]$-equivariant section
			of the map $B \to C$. Let $P \setminus G = \{Pg_1, \ldots, Pg_m\}$,
			where $p \nmid m = [G:P]$. Then, as one easily checks
			\[ \wt s : C \to B, \quad \wt s(x) :=
			\frac 1m \sum_{i = 1}^m g_i^{-1} s(g_i x) \]
			is a $k[G]$-equivariant section of $B \to C$.
		\end{proof}		
		Unfortunately we are able to answer Question~\ref{q:splitting_k[G]-mod} only for
		the class of groups that have 'tame' modular representation theory, i.e.
		groups with a cyclic $p$-Sylow subgroup.
		\begin{Lemma} \label{lem:cyclic_sylow}
		 Suppose that $k$ is a field of characteristic $p > 0$ and $G$ is a finite group with a cyclic $p$-Sylow subgroup.
		 Let
		 \begin{equation} \label{eqn:ABC_exact_sequence_2}
		  0 \to A \to B \to C \to 0
		 \end{equation}
		 be an exact sequence of $k[G]$-modules. If the sequence
		 \begin{equation*} 
		  0 \to A^G \to B^G \to C^G \to 0
		 \end{equation*}
		 is exact on the right then the exact sequence~\eqref{eqn:ABC_exact_sequence_2}
		 splits $G$-equivariantly.
		\end{Lemma}
		\begin{proof}
		 Without loss of generality we can assume that $G = \ZZ/p^n$ is a cyclic $p$-group
		 (by Lemma~\ref{lem:splitting_sylow}). Note that $k[\ZZ/p^n] \cong k[x]/(x - 1)^{p^n}$. The classification theorem of finitely generated modules over the principal ideal 
		domain $k[x]$ (cf. \cite[Theorem 12.1.5]{DummitFoote2004}) implies that every finitely generated indecomposable $k[\ZZ/p^n]$-module is of the form:
		\[ J_i = k[x]/(x - 1)^i \qquad \textrm{ for some } i = 1, \ldots, p^n. \]
		Denote also $J_0 := 0$. Using Smith's Normal Form theorem (cf. \cite[Theorem 12.1.4]{DummitFoote2004}) we obtain a commutative diagram:
		 \begin{center}
		  \begin{tikzcd}
				A \arrow[r, hook] \arrow[d, "\cong" description] & B \arrow[d, "\cong" description] \\
				\bigoplus_{i=1}^l J_{a_i} \arrow[r, hook] & \bigoplus_{i=1}^m J_{b_i}
			\end{tikzcd}
		 \end{center}
		 where $l \le m$, $a_i \le b_i$ and $J_{a_i} \hookrightarrow J_{b_i}$ is the natural inclusion.
		 Hence we are reduced to proving the claim for the exact sequence:
		 \[ 0 \to J_a \to J_b \to J_c \to 0, \]
		 where $a + b = c$, $0 \le a, b, c \le p^n$. However, the equality
		 \begin{eqnarray*}
		  \dim_k J_s^G =
		  \begin{cases}
		   1, & \textrm{ if } s \neq 0\\
		   0, & \textrm{ otherwise.}
		  \end{cases}
		 \end{eqnarray*}
			makes it obvious that $a = 0$ or $c = 0$. This finishes the proof.
		 
		\end{proof}

		\subsection{Relation to the problem of lifting coverings}
		Let $X$ be a smooth variety over an algebraically closed field $k$ of characteristic~$p$.
		Denote by~$X'$ the Frobenius twist of $X$ and by $F : X \to X'$ -- the absolute Frobenius morphism of $X$. Recall that in this case we have the following Cartier isomorphism (cf.~\cite{Cartier_nouvelle_operation})
		of $\mc O_{X'}$-modules:
		\begin{equation}
		 \mc C^{-1} : \Omega_{X'}^i \to h^i(F_{*} \Omega_X^{\bullet}).
		\end{equation}
        Therefore the spectral sequence~\eqref{eqn:conjugated_spectral_sequence_hypercohomology}
        for the de Rham cohomology becomes:
        \begin{equation} \label{eqn:conjugated_spectral_sequence_cartier}
         _{II} E^{ij} = H^i(X', \Omega_{X'}^j) \Rightarrow H^{i+j}_{dR}(X/k).
        \end{equation}
        Let for any $k$-vector space $V$, $V'$ denote the $k$-vector space
        with the same underlying abelian group as $V$ and the scalar multiplication $(\lambda, v) \mapsto \lambda^p \cdot v$.
        Then one easily checks that:
        \[
         H^1(X', \mc O_{X'}) \cong H^1(X, \mc O_X)', \quad
         H^0(X', \Omega_{X'/k}) \cong H^0(X, \Omega_{X/k})'
        \]
        Therefore if the spectral sequence~\eqref{eqn:conjugated_spectral_sequence_cartier} degenerates on
        the second page, we obtain the following exact sequence:
        \begin{equation} \label{eqn:conjugate_hodge_de_rham_se}
         0 \to H^1(X, \mc O_X)' \to H^1_{dR}(X/k) \to H^0(X, \Omega_{X/k})' \to 0.
        \end{equation}

        Suppose now that $X$ is equipped with an action of a finite group $G$.
		We say that the pair $(X, G)$ lifts to $W_2(k)$,
		if there exists a smooth scheme $\bb X$ over $W_2(k)$ and a homomorphism $G \to \Aut_{W_2(k)}(\bb X)$ such that
		\[
			(\bb X, G \rightarrow \Aut_{W_2(k)}(\bb X)) \times_{W_2(k)} k = (X, G \rightarrow \Aut_k(X)).
		\]
		The following Theorem is a $G$-equivariant version of the main result of
		\cite{DeligneIllusie_relevements} and follows from the functoriality
		of the result of Deligne and Illusie.
		\begin{Theorem} \label{thm:deligne_illusie}
		 Suppose that the pair $(X, G)$ lifts to $W_2(k)$ and that $\dim X < p$. Then the exact
		 sequence~\eqref{eqn:conjugate_hodge_de_rham_se} of $k[G]$-modules splits.
		 In particular, there is an isomorphism \eqref{eqn:intro_decomposition} of $\FF_p[G]$-modules.
		\end{Theorem}
		\begin{proof}
		 The article~\cite{DeligneIllusie_relevements} provides for each lift $\wt X/W_2(k)$ of $X/k$
		 an isomorphism:
		 \begin{equation} \label{eqn:deligne_illusie_varphi}
		  \varphi_{\wt X} : \bigoplus_i \Omega_{X'/k}^i[-i] \to F_{*} \Omega_{X/k}^{\bullet}
		 \end{equation}
         in $D(X')$, the derived category of coherent $\mc O_{X'}$-modules.
         Recall that $\varphi_{\wt X}^0$ is defined by:
         \[
          \mc O_{X'}[0] \stackrel{\mc C^{-1}}{\rightarrow} h^0(F_{*} \, \Omega_{X/k}^{\bullet})[0]
          \hookrightarrow F_{*} \, \Omega_{X/k}^{\bullet}.
         \]
         Thus, by applying
         the first cohomology to~\eqref{eqn:deligne_illusie_varphi} we obtain an isomorphism:
         \begin{equation} \label{eqn:splitting_of_con_hodge_de_rham_se}
          \phi_{\wt X} : H^0(X', \Omega_{X'/k}^1) \oplus H^1(X', \mc O_{X'}) \to H^1(X', F_{*} \, \Omega_{X/k}^{\bullet})
          \cong H^1(X, \Omega_{X/k}^{\bullet})
         \end{equation}
         which yields a splitting of~\eqref{eqn:conjugate_hodge_de_rham_se}.
         Now, observe that $\varphi_{\wt X}$ and $\phi_{\wt X}$ are functorial with respect to $\wt X$.
         Thus, if $(X, G)$ lifts to $W_2(k)$,~\eqref{eqn:splitting_of_con_hodge_de_rham_se} becomes an 
         isomorphism of $k[G]$-modules. 
		\end{proof}
        \begin{Remark}
        	If $G$ is a cyclic $p$-group and $V$ is a $k[G]$-module with $\dim_k V {<} \infty$,
        	one may easily prove that $V \cong V'$ as $k[G]$-modules.
        \end{Remark}
        The following important question remains open.
        \begin{Question}
        	Suppose that the pair $(X, G)$ lifts to $W_2(k)$. Does it follow that the exact sequence of
        	$k[G]$-modules~\eqref{eqn:hodge_de_rham_se} splits?
        \end{Question}
        \begin{Proposition} \label{prop:main_theorem_conjugate_sequence}
        	Keep the notation introduced in Section~\ref{sec:intro}. If the exact sequence~\eqref{eqn:conjugate_hodge_de_rham_se}
        	of $k[G]$-modules splits, then the action of $G$ on $X$ is weakly ramified.
        \end{Proposition}
    	\begin{proof}
    		Note that for a $k[G]$-module $V$ of finite $k$-dimension, $\dim V^G = \dim (V')^G$. Thus the proof follows
    		by the method of proof of Main Theorem.
    	\end{proof}
		\noindent The following is an immediate consequence of Theorem~\ref{thm:deligne_illusie}
		and Proposition~\ref{prop:main_theorem_conjugate_sequence}.
		\begin{Corollary} \label{cor:lifts_then_weakly}
		 Suppose that $p > 2$, $X$ is a smooth projective curve over $k$ and the pair $(X, G)$
		 lifts to $W_2(k)$. Then the action of $G$ on $X$ is weakly ramified.
		\end{Corollary}
		%
		\noindent Note that it was known previously that non-weakly ramified actions
		on curves do not lift to $W(k)$ (cf. \cite[Corollary, Sec. 4]{Nakajima_Automorphism}).\\
		
		The problem of lifting Galois coverings of curves from characteristic $p$
		to $W_2(k)$ and $W(k)$ has been studied extensively
		in the literature, cf. e.g.~\cite{Pop_OortConjecture} for the case $G = \ZZ/p$. In particular,
		it is possible to classify all weakly ramified group actions
		into liftable and non-liftable ones (cf. \cite[Section 4.2]{BertinMezard2000}
		and \cite[Section 4.1]{CornelissenKato2003}).
		However, we weren't able to extract any information that would help us to understand
		the behaviour of the sequence~\eqref{eqn:hodge_de_rham_se}
		for curves with weakly ramified group action.

		\begin{Corollary} \label{cor:ordinary}
			Suppose that a finite group $G$ acts on an ordinary curve~$X$.
			Then the exact sequences~\eqref{eqn:hodge_de_rham_se} and~\eqref{eqn:conjugate_hodge_de_rham_se}
			split $G$-equivariantly and
			\[
				H^0(X, \Omega_{X/k})' \cong H^0(X, \Omega_{X/k}), \quad H^1(X, \mc O_X)' \cong H^1(X, \mc O_X).
			\]
			as $k[G]$-modules.
		\end{Corollary}
		\begin{proof}
			Let $A$ be the Jacobian variety of $X$. Observe that the Abel-Jacobi map induces an
			isomorphism between the Hodge-de Rham sequences of $X$ and $A$ (cf. \cite[Proposition III.2.1, Lemma III.9.5.]{Milne_AVs}).
			The same applies to the conjugate Hodge-de Rham sequences. Moreover, $A$ is ordinary, and thus the natural
			inclusions:
			\[
				H^0(A, \Omega_{A/k}) \to H^1_{dR}(A/k), \quad H^1(A, \mc O_A)' \to H^1_{dR}(A/k)
			\]
			induce an isomorphism $H^1_{dR}(A/k) \cong H^0(A, \Omega_{A/k}) \oplus H^1(A, \mc O_A)'$ (cf. \cite[\S 2.1]{WedhornDeRham}).
			This isomorphism is clearly functorial and thus is an isomorphism of $k[G]$-modules.
			The remaining statement is clear.
		\end{proof}		 
		Note in particular that Corollary~\ref{cor:ordinary} implies that ordinary curves admit only weakly
		ramified group actions. This follows also from the Deuring-Shafarevich formula (cf. \cite{Subrao_p_rank}).\\
		
		\appendix
		\section{Computing the dimension of $H^1(X, \mc O_X)^G$} \label{sec:H^1(O_X)^G}
	For completeness we include also the following proposition, which 
	allows in many situations to compute dimensions of $H^0(X, \Omega_{X/k})^G$, 
	$H^1(X, \mc O_X)^G$ and $H^1_{dR}(X/k)^G$ in terms of invariants of $Y$ and 
	group cohomology of sheaves. Note that by~Corollary~\ref{cor:pi_G_omega} and Proposition~\ref{prop:main_lemma} we are left with computing the dimension of $H^1(X, \mc O_X)^G$.
	\begin{Proposition} \label{prop:invariants}
		In the notation of Section~\ref{sec:intro} suppose that there exists $Q_0 \in Y$ such that
			      \[ p \nmid \# \pi^{-1}(Q_0). \]
			      Then:
			      \begin{eqnarray*}
			      	\dim_k H^1(X, \mc O_X)^G &=& g_Y + \sum_{Q \in Y} H^1(G, (\pi_* \mc O_X)_Q) \\
			      	&-& \dim_k H^1(G, k).
			      \end{eqnarray*}
	\end{Proposition}
\begin{proof}
		By substituting $\mc F^0 = \pi_* \mc O_X$ in the formula~\eqref{eqn:invariants_of_cohomology}
		and using Lemma~\ref{lem:monomorphism_R^2} and Corollary~\ref{cor:commutative_diagram} it suffices to prove that the natural map
		\begin{equation} \label{eqn:map_H^2(G, k)}
			H^2(G, k) \cong H^2(G, H^0(Y, \pi_* \mc O_X)) \to 
			H^0(Y, \mc H^2(G, \pi_* \mc O_X))		 
		\end{equation}
		is injective. One easily sees that 
		\begin{eqnarray*}
			\mc H^2(G, \pi_* \mc O_X) &\cong& \bigoplus_{Q \in Y}
			i_{Q, *} \bigg( H^2(G, (\pi_* \mc O_X)_Q) \bigg)
		\end{eqnarray*}
		is a direct sum of skyscraper sheaves. Choose any $P_0 \in \pi^{-1}(Q_0)$.
		Observe that by Lemma~\ref{lem:group_cohomology_higher_ramification} we have: 
		\[ H^2(G, (\pi_* \mc O_X)_Q) \cong H^2(G_{P_0, 1}, \mc O_{X, P_0}). \]
		But $\mc O_{X, P_0} \cong k \oplus \mf m_{X, P_0}$ as
		a $k[G_{P_0, 1}]$-module and therefore
		\[ H^2(G_{P_0, 1}, \mc O_{X, P_0})
		\cong H^2(G_{P_0, 1}, k) \oplus H^2(G_{P_0, 1}, \mf m_{X, P_0}). \]
		One easily sees that the map~\eqref{eqn:map_H^2(G, k)} factors as
		\[
			H^2(G, k) \to H^2(G_{P_0, 1}, k) \hookrightarrow H^2(G_{P_0, 1}, k) \oplus
			H^2(G_{P_0, 1}, \mf m_{X, P_0}).
		\]
		where the first map is the restriction map $\res_{G_{P_0, 1}}^G$.
		Now note that $p \nmid \# \pi^{-1}(Q_0) = [G : G_{P_0}]$ and
		thus $G_{P_0, 1}$ is a $p$-Sylow subgroup of $G$ by
		\cite[Corollary 4.2.3., p. 67]{Serre1979}. Thus
		by~\eqref{eqn:sylow} $\res_{G_{P_0, 1}}^G$ is an isomorphism.
		This ends the proof.
	\end{proof}
\begin{Example}
 Keep the notation introduced in Section~\ref{sec:intro} and suppose that $G \cong \ZZ/p$.
 Then by Lemma~\ref{lem:nP_of_AS}, one has $d_Q = (n_Q + 1) \cdot p$ for all $Q \in Y$
 and therefore:
 \begin{equation} \label{eqn:Rprim_for_AS}
		R' = \sum_{Q \in Y} \left[\frac{(n_Q + 1) \cdot (p-1)}{p}\right] (Q).
 \end{equation}
 Moreover, by Lemma~\ref{lem:H^1(I)}
	\begin{equation}
		\dim_k H^1(G, (\pi_* \mc O_X)_Q) = \left[\frac{(n_Q + 1) \cdot (p-1)}{p}\right].
	\end{equation} 
 Suppose that the action of $G$ on $X$ is not free.
 Then by Corollary~\ref{cor:pi_G_omega}, Proposition~\ref{prop:invariants} and \eqref{eqn:Rprim_for_AS} we obtain:
 \begin{eqnarray*}
  \dim_k H^0(X, \Omega_{X/k})^G &=& \dim_k H^1(X, \mc O_X)^G\\
		&=&	g_Y - 1 + \sum_{Q \in Y} \left[\frac{(n_Q + 1) \cdot (p-1)}{p}\right].
 \end{eqnarray*}
	Moreover, by previous computations and by Proposition~\ref{prop:main_lemma} we obtain:
	\begin{eqnarray*}
	 \dim_k H^1_{dR}(X/k)^G &=& 2 (g_Y - 1)\\
	 &+& \sum_{Q \in Y} \bigg(\left[\frac{(n_Q + 1) \cdot (p-1)}{p}\right] + 1 + \left[\frac{n_Q - 1}{p}\right]\bigg).
	\end{eqnarray*}
	If the action of $G$ is free, then a similar reasoning leads to the formulas:
	\begin{eqnarray*}
	 \dim_k H^0(X, \Omega_{X/k})^G &=& \dim_k H^1(X, \mc O_X)^G = g_Y\\
	 \dim_k H^1_{dR}(X/k)^G &=& 2 g_Y.
	\end{eqnarray*}
\end{Example}
		
		\bibliography{bibliografia}

\end{document}